\documentclass[11pt,a4paper,twoside]{amsart}
\usepackage{amsmath,amssymb,amsfonts,amsthm}
\usepackage{hyperref}

\def\pa{\partial}

\def\to{\longrightarrow}
\def\mto{\longmapsto}
\def\a{\alpha}
\def\b{\beta}

\def\g{\gamma}
\def\p{\phi}
\def\P{\Phi}
\def\G{\Gamma}
\def\s{\psi}

\def\eps{\epsilon}

\def\o{\circ}

\def \lim{\varprojlim}

\def \D{{\mathcal{D}}}

\def\N{\mathbb{N}}

\def\F{\mathbb{F}}
\def\R{\mathbb{R}}
\def\E{\mathbb{E}}

\newtheorem{The}{Theorem}[section]
\newtheorem{Pro}[The]{Proposition}

\newtheorem{Cor}[The]{Corollary}
\theoremstyle{definition}
\newtheorem{Def}[The]{Definition}
\newtheorem{Rem}[The]{Remark}
\newtheorem{Examp}[The]{Example}

\begin{document}
\title{ Second order time dependent tangent bundles  and their  applications}
\author{Ali Suri }
\address{Department  of Mathematics, Faculty of sciences \\
Bu-Ali Sina University, Hamedan 65178, Iran.}
\email{ali.suri@gmail.com \& a.suri@math.iut.ac.ir \& a.suri@basu.ac.ir}
\maketitle {\hspace{2.5cm}}

\begin{abstract}
The aim of this paper is to geometrize time dependent Lagrangian mechanics in a way that the framework of second order tangent bundles plays an essential role. To this end, we first introduce the concepts of time dependent connections and time dependent semisprays on a manifold $M$ and their induced vector bundle structures  on the second order time dependent tangent bundle $\R\times T^2M$. Then we turn our attention to regular time Lagrangians and their interaction with $\R\times  T^2M$ in different  situations such as mechanical systems with potential fields, external forces and holonomic constraints. Finally we propose an examples to support our theory.

\end{abstract}

\pagestyle{headings} \markright{Second order time dependent tangent bundles}
\tableofcontents

\section{Introduction}
The time dependent tangent bundle and time dependent second order ordinary differential equations arise naturally in many geometric and physical situations (for a comprehensive treatment in the finite dimensional case  and for references see \cite{Mir-Anas}).

On the other hand the geometry of higher order tangent  spaces and higher order Mechanics has been extensively  studied by many authors like  Crampin etc. \cite{crampin}, Dodson and Galanis \cite{Dod-Gal}, Bucataru \cite{Ioan, Ioan-2}, De Le$\acute{\textrm{o}}$n and Rodriguez \cite{Leon, Leon 2},   Miron \cite{Miron} and Popescu \cite{popescu}.

In this paper we follow a different line that is \emph{geometrization  of time dependent Lagrangian mechanics in which the framework of second order time dependent tangent bundles plays an essential role}. More precisely, we  show that  $\pi_2:\R\times T^2M\to \R\times M$ which arises naturally as an extension of ordinary tangent bundle,  with a suitable choice of   vector bundle structure, prepares a rich geometric  framework which carries many geometric and mechanical properties. For example we show that the integral curves of the zero section $\xi\in\G(\pi_2)$ are motions of a time dependent  mechanical system or solutions of a time dependent second order ordinary differential equations.

In section \ref{section preliminaries} we introduce the concepts of time dependent  connections and time dependent semisprays and the relations between them. Then we show that at the presence  of a time dependent semispray (or equivalently a time dependent connection), $\pi_2:\R\times T^2M\to \R\times M$ admits a vector bundle structure isomorphic to $\R\times TM\oplus \R\times TM$. The converse of the above fact is also true in the sense that a vector bundle structure on $\R\times T^2M$ which makes it isomorphic to two copies of time dependent tangent bundle, declares a time dependent semispray  (or equivalently a time dependent connection) on $M$.

Then we turn our attention to   the concepts of time dependent (regular) Lagrangian and Lagrangian vector fields in the general case of Banach manifolds. As a fundamental step we derive the semispray (second order vector field) of a regular time dependent  Lagrangian as well as the case of Lagrangian systems with time dependent external forces, potential field and holonomic constraints. Then we declare the induced geometric structures on $(\pi_2,\R\times T^2M,\R\times M)$  and we derive the equations of motion as integral curves of the zero section of this vector bundle.

Afterward we reveal the relations between invariant Lagrangians and related semisprays and vector bundle isomorphisms on the second order time dependent tangent bundles. More precisely we show that
invariant Lagrangians induce invariant vector bundle structures (up to isomorphism) on $\R\times  T^2M$.

Then, as an application, we discuss the case of group of $H^s$ volume preserving diffeomorphisms $\D^s_\mu$ on  a compact Riemannian manifold.  Subsequently, we derive the time dependent semispray and the induced vector bundle structure on $\R\times T^2\D^s_\mu$ which describes the motion of an incompressible fluid at the presence of a time dependent external force in $M$.

Trough this paper all the maps and manifolds are assumed to be smooth,
but a lesser degree of differentiability can be assumed. Whenever partition of unity is necessary, we assume that our manifolds are partitionable \cite{Lang, ali}.

Finally we note that, all the results can be modified for the autonomous and finite dimensional cases  directly.

%
\section{Preliminaries}\label{section preliminaries}
In this section first we summarize  the necessary preliminary material that we need for a self contained paper.
Then we introduce the concepts of time dependent (possibly nonlinear) connections and time dependent semisprays and second order time dependent tangent bundle $\pi_2:\R\times T^2M\to \R\times M$.

Let $M$ be a manifold modeled on the Banach space (possibly infinite dimensional) $\E$ and $\pi_M:TM\to M$
be its tangent bundle.  Define
\begin{eqnarray*}
\pi:\R\times TM&\to &\R\times M\\
(t,[\g,x])  &\mto & (t,x)
\end{eqnarray*}
and  $V_{(t,[\g,x])}\pi=ker (d_{(t,[\g,x])}\pi)$ which is known as the vertical subspace and  locally
\begin{eqnarray*}
V_{(t,[\g,x])}\pi&=&\{(t,x,y;s,a,b)~~;~~d\pi(t,x,y;s,a,b)=(t,x;0,0)\}\\
&=&\{(t,x,y;0,0,b)~~;~~t\in\R ~\textrm{and}~x\in U, ~y,b\in \E\}
\end{eqnarray*}
\begin{Def}
A nonlinear connection on $E:=\R\times TM$ is a smooth distribution $N:u\in E\mto N_u\subseteq T_uE$ for which
\begin{equation*}
T_uE=V_u\pi\oplus N_u.
\end{equation*}
\end{Def}
Consider the atlases $\mathcal{A}=\{(U_\a,\p_\a)\}_{\a\in I} $ and $\mathcal{B}=\{(\overline{U}_\a,\s_\a)\}_{\a\in I} $ for $M$ and $E$ respectively where $\overline{U}_\a=\pi^{-1}(U_\a)$ and $\s_\a(t,[\g,x])=\big(  t,  (\p_\a\o\g)(0),  (\p_\a\o\g)'(0)  \big)$ for $x\in U_\a$. Let $\pi_E:TE\to E$ be the tangent bundle of $E$. Then,  $\mathcal{B}$ in a canonical way  induces the atlas $\mathcal{C}=\{(\widetilde{U}_\a,\P_\a)\}_{\a\in I} $ with $\widetilde{U}_\a=\pi_E^{-1}(\overline{U}_\a)$ and $\P_\a=T\s_\a$, $\a\in I$.

Let $\nu:TE\to V\pi$ be the vertical projection. For any $\a\in I$ set $\nu_\a=\P_\a\o\nu\o\P_\a^{-1}$. Since $\nu_\a\o\nu_\a=\nu_\a$, it follows that
$\nu_\a(t,x,y;0,0,b)=(t,x,y;0,0,b)$. As a consequence
\begin{eqnarray*}
\nu_\a:\R\times U_\a\times \E\times \R\times\E^2 &\to &    \R\times U_\a\times \E\times\{0_\R\}\times\{0_\E\}\times \E\\
(t,x,y;s,a,b)   &\mto&    (t,x,y;0,0,b+N_\a(t,x,y)(s,a))
\end{eqnarray*}
for a smooth map $N_\a:\R\times U_\a\times \E\to\mathcal{L}^2(\R\times\E,\E)$ where $\mathcal{L}^2(\R\times\E,\E)$ is the space of all bilinear continuous maps form $\R\times \E$ to $\E$. We add here to the convention that
\begin{eqnarray*}
N_\a(t,x,y)(s,a):=N^0_\a(t,x,y)s+N_\a^1(t,x,y)a
\end{eqnarray*}
where $N_\a^0:\R\times U_\a\times \E\to\mathcal{L}(\R,\E)$ and $N_\a^1:\R\times U_\a\times \E \to\mathcal{L}(\E,\E)$.

The maps $\{N_\a^0,N_\a^1\}_{\a\in I}$ are called the coefficients of the nonlinear connection. For any $\a,\b\in I$ with $U_{\a\b}:=U_\a\cap U_\b\neq \emptyset$, the compatibility condition (change of coordinates) for $ N_\a^i$ and $ N_\b^i$, $ i=0, 1$, comes form the fact that
\begin{eqnarray*}
\nu_\b\o\P_\b\o\P_\a^{-1}=\P_\b\o\P_\a^{-1}\o\nu_\a.
\end{eqnarray*}
After some calculations, for any $x\in U_{\a\b}$, $t,s\in \R$ and $y,a,b\in\E$, we get
\begin{eqnarray}\label{equation compatibility cond for a conn map}
d\p_{\b\a}(x)N_\a(t,x,y)(s,a) &=&
N_\b \big(t,\p_{\b\a}(x),d\p_{\b\a}(x)y \big)   \big(  s,d\p_{\b\a}(x)a  \big)\\
&&\nonumber+d^2\p_{\b\a}(x)(a,y).
\end{eqnarray}
Setting $a=0$  we see that
\begin{eqnarray}\label{eq comp N0}
d\p_{\b\a}(x)N_\a^0(t,x,y)s
=N_\b^0 \big(t,\p_{\b\a}(x),d\p_{\b\a}(x)y \big)s
\end{eqnarray}
and for $s=0$ the equation (\ref{equation compatibility cond for a conn map}) implies that
\begin{eqnarray}\label{eq comp N1}
d\p_{\b\a}(x)N_\a^1(t,x,y)a &=& N_\b^1 \big(t,\p_{\b\a}(x),d\p_{\b\a}(x)y \big)d\p_{\b\a}(x)a\\
&&\nonumber + d^2\p_{\b\a}(x)(a,y).
\end{eqnarray}

\begin{Def}
A time dependent semispray $S$ on $\R\times M$ is a vector field on $\R\times TM$ such that locally
\begin{equation*}
S_\a(t,x,y)=(t,x,y,1,y,-G_\a(t,x,y))
\end{equation*}
where $S_\a=:\P_\a\o S\o\p_\a^{-1}$ and  $G_\a:\R\times U_\a\times \E\to \E$, $\a\in I$, are called the local components (coefficients) of the semispray $S$.
\end{Def}
For $\a,\b\in I$ with $U_{\a\b}\neq\emptyset$, the compatibility condition for $G_\a$ and $G_\b$ is obtained from the equality
\begin{eqnarray*}
S_\b\o(\p_\b\o\p_\a^{-1})=(\P_\b\o\P_\a^{-1})\o S_\a
\end{eqnarray*}
and it is
\begin{eqnarray}\label{eq compatibility condition for semisprays}
G_\b(t,\p_{\b\a}(x),d\p_{\b\a}(x)y)+d^2\p_{\b\a}(x)(y,y)=d\p_{\b\a}(x) G_\a(t,x,y).
\end{eqnarray}
We remind that $c:(-\eps,\eps)\to M$ is called a geodesic of $S$ if $(t,c''(t))=S(t,c'(t))$ or equivalently, for any $\a\in I$ which the image of $c$ meets $U_\a$, the following second order differential equation
\begin{eqnarray*}
(\p_\a\o c)''(t)+G_\a\big( t, (\p_\a\o c)(t),(\p_\a\o c)'(t) \big)=0
\end{eqnarray*}
is satisfied.

\subsection{Second order time dependent tangent bundles}
As a natural extension of the  tangent bundle $TM$ we define the
following equivalence relation. The curves  $\g_1,\g_2 \in C_{x_0}:=\{\g:(-\eps,\eps)\to M~;~ \g(0)=x_0  ~ \textrm{and}~ \g
$ $~\textrm{is smooth }\}$
are said to be $2$-equivalent, denoted by $\g_1\approx_{x_0}^2\g_2$,
if and only if $\g_1^{\prime}(0)=\g_2^{\prime}(0)$  and $\g_1^{\prime\prime}(0)=\g_2^{\prime\prime}(0)$.
Define ${T}^2_{x_0}M:=C_{x_0}/\approx_{x_0}^2$ and the \textbf{
second order tangent bundle } of M to be
${T}^2M:=\bigcup_{x\in M}{T}^2_{x}M$. Denote by $[\g,{x_0}]_2$ the
representative of the equivalence class containing $\g$  and define
the canonical projection ${\pi}_2:T^2M\to M$ which projects
$[\g,{x_0}]_2$ onto $ x_0$. The second order time dependent tangent bundle of $M$ is
\begin{eqnarray*}
id_\R\times \pi_2:\R\times T^2M   &\to&    \R\times M \\
(t,{[\gamma,{x}]_2})&\mto&  (t,x )
\end{eqnarray*}
By abuse of notation  the same symbol  $\pi_2$ is used for  $id_\R\times \pi_2$.
Note that, in contrast to the $(\tau_M,TM,M)$ and $(\pi,\R\times TM,\R\times TM)$, the fibrations $(\pi_2,T^2M,M)$ and $(\pi_2,\R\times T^2M,\R\times M)$ do not admit vector bundle structures generally \cite{Dod-Gal, ali, Suri Osck}.

%
%
%
%
%
%
However,  next theorem reveals the geometric structure which a time dependent semispray (or a time dependent connection) induces on $\pi_2$.
\begin{The}\label{theo vb st by spray}
\textbf{i}. Let $S$ be a time dependent semispray on $\R\times M$. Then the family of trivializations
\begin{eqnarray*}
\P_\a^2:\pi_2^{-1}(\R\times U_\a)  &\to&   \R\times U_\a\times \E\times \E\\
(t,[\g,x]_2)   &\mto&   \Big(    t, (\p_\a\o\g)(0) ,(\p_\a\o\g)'(0), (\p_\a\o\g)''(0) \\
&&+G_\a        \big(t,(\p_\a\o\g)(0), (\p_\a\o\g)'(0)\big)   \Big)
\end{eqnarray*}
defines a vector bundle structure on $\pi_2:\R\times T^2M\to M$ isomorphic to $\R\times TM\oplus \R\times TM$ and the structure group $GL(\E\times \E)$.

Moreover, integral curves of the zero section of this bundle are geodesics of the semispray $S$.

\textbf{ii.} Let $\pi_2:\R\times T^2M\to M$ possesses a vector bundle structure isomorphic  to $\R\times TM\oplus \R\times TM$ and the structure group $GL(\E\times \E)$. Then on $M$ a time dependent semispray  can be defined.
\end{The}
\begin{proof}
\textbf{i.}
Clearly $\P_\a^2$ is well defined and  injective. Moreover for any $(t,x,y,z)\in\R\times U_\a\times\E^2$ the second order tangent vector
$(t,[\g,x]_2)$ with $\p_\a\o\g(h)=\p_\a(x)+hy+\frac{h^2}{2}(z-G_\a(t,x,y))$ is mapped onto $(t,x,y,z)$ via $\P_\a^2$. Finally for any $(t,x,y,z)\in\R\times U_{\a\b}\times \E^2$ we have
\begin{eqnarray*}
&&\P_\b^2\o{\P_\a^2}^{-1}(t,x,y,z)=\\
&=&\Big(   t,(\p_\b\o\g)(0)    ,   (\p_\b\o\g)'(0)   ,   (\p_\b\o\g)''(0)  \\
&&+G_\b        \big(t,(\p_\b\o\g)(0), (\p_\b\o\g)'(0)\big) \Big)\\
&=&\Big(   t,\p_{\b\a}(x),(\p_{\b\a}\o\p_\a\o\g )'(0),   (\p_{\b\a}\o\p_\a\o\g)''(0)\\
&&+G_\b        \big(t,(\p_{\b\a}\o\p_\a\o\g)(0), (\p_{\b\a}\o\p_\a\o\g)'(0)\big)   \Big)\\
&=&\Big(   t   ,   \p_{\b\a}(x)  ,    d\p_{\b\a}(x)y, d\p_{\b\a}(x)\big(z-G_\a(t,x,y)\big)    +  d^2\p_{\b\a}(x)(y,y)\\
&&+G_\b        \big(t,x, (\p_{\b\a}\o\p_\a\o\g)'(0)\big)   \Big)\\
&\stackrel{*}{=}&\Big(   t   ,   \p_{\b\a}(x)  ,    d\p_{\b\a}(x)y, d\p_{\b\a}(x)z  \Big)
\end{eqnarray*}
where in $*$ we used (\ref{eq compatibility condition for semisprays}). As a result
the transition map
\begin{eqnarray*}
\P_{\b\a}^2:\R\times U_{\b\a}       &\to&       \mathbb{GL}(\E\times \E)\\
(t,x)   &\mto&  \big(   d\p_{\b\a}(x), d\p_{\b\a}(x)  \big)
\end{eqnarray*}
is a  smooth map  and $\P_{\b\a}^2$ does not depend on time.

Now consider the zero section $\xi:\R\times M\to \R\times T^2M$ which maps $(t,x)$ to $(t,x,0,0)\in \R\times T^2_{x}M$. Following \cite{2de}, theorem 2.2 it is easy to see that   $c:(-\eps,\eps)\to M$ is an integral curve of $\xi$ if and only if, for any $\a\in I$, it satisfies  the equation
\begin{equation*}
(\p_\a\o c)''(t)+G_\a\big(  (\p_\a\o c)(t),(\p_\a\o c)'(t)  \big)=0
\end{equation*}
as we required.

\textbf{ii.}
Let $\{   \big( \pi_2^{-1}(\R\times U_\a),\P^2_\a  \big)\}_{\a\in I}$ be a family of trivializations for $\pi_2$. Then, for any $x\in U_\a$,
\begin{equation*}
\P_{\a,x}^2:=\P_{\a}^2|_x:\pi_2^{-1}(t,x) \pi^{-1}(t,x)\oplus\pi^{-1}(t,x)
\end{equation*}
is a linear isomorphism. Consider the family of charts for $M$ such that $d\p_\a(x)=proj_1\p\o \P_{\a,x}^2$ where $proj_1$ is projection to the first factor. (see also \cite{Dod-Gal} and \cite{ali}).
Now, set
\begin{equation*}
G_\a(t,x,y)=proj_4\o\P_\a^2(t,[\g,x]_2)-(\p_\a\o \g)''(0)
\end{equation*}
where $\bar{\g}(s)=x+sy$ and $\g=\p_\a^{-1}\o\bar\g$. Then we have
\begin{eqnarray*}
G_\b(t,\p_{\b\a}(x),d\p_{\b\a}(x)y)=proj_4\o\P_\b^2   (t,[\theta,x]_2)-(\p_\b\o\theta)''(0)
\end{eqnarray*}
where $\bar\theta=\p_{\b\a}\o\bar\g$ and $\theta=\p_\b^{-1}\o\bar\theta$. Note that $[\theta,x]_2=[\g,x]_2$. However
\begin{eqnarray*}
G_\b(t,\p_{\b\a}(x),d\p_{\b\a}(x)y)   &=&    proj_4\o\P_\b^2(t,[\theta,x]_2)-(\p_\b\o\theta)''(0)\\
&& \hspace{-20mm}    =proj_4\o\P_\b^2\o{\P_\a^2}^{-1}\o\P_\a^2(t,[\theta,x]_2)-(\p_{\b\a}\o\g)''(0)\\
&& \hspace{-20mm}    =d\p_{\b\a}(x) \Big(  proj_4\o\P_\a^2(t,[\g,x]_2)   \Big)  -d^2\p_{\b\a}(x)(y,y)\\
&& \hspace{-20mm}    =    d\p_{\b\a}(x) \Big(  proj_4\o\P_\a^2(t,[\g,x]_2)   \Big)  -d^2\p_{\b\a}(x)(y,y)\\
&& \hspace{-20mm}    =  d\p_{\b\a}(x) G_\a(t,x,y)-d^2\p_{\b\a}(x)(y,y)
\end{eqnarray*}
that is $G_\a$ and $G_\b$ satisfy the compatibility condition (\ref{eq compatibility condition for semisprays}). As a consequence the family $\{G_\a\}_{\a\in I}$ defines a (time dependent) semispray on $M$.
\end{proof}
%
%
%
\begin{Rem}
One can replace the zero section in the above theorem with any section $\eta$ of the vector bundle $\pi_2:\R\times T^2M\to \R\times M$ with the property $proj_3\o\P_\a^2\o\eta=proj_3\o\P_\a^2\o\eta$, $\a\in I$, and prove the same result (Here $proj_3$ stands for the projection to the third factor). More precisely, the integral curves of the section $\eta$ are geodesics of the semispray $S$ too. (See also \cite{2de}, section 5.1.)
\end{Rem}
%
%
%
The following proposition  can be proved as theorem \ref{theo vb st by spray}.
\begin{Pro}\label{the vb strucures}
The followings hold true.\\
\textbf{i}. Let $N$ be a time dependent nonlinear connection on $M$. Then the family of trivializations
\begin{eqnarray}\label{trivialization}
\nonumber \P_\a^2:\pi_2^{-1}(U_\a)  &\to&   \R\times U_\a\times \E\times \E\\
(t,[\g,x]_2)   &\mto&   \Big(    t, (\p_\a\o\g)(0) ,(\p_\a\o\g)'(0), (\p_\a\o\g)''(0) \\
\nonumber&&+\underbrace{N_\a        \big(t,(\p_\a\o\g)(0), (\p_\a\o\g)'(0)\big)[1,(\p_\a\o\g)'(0)]}_{P_\a}   \Big)
\end{eqnarray}
$\a\in I$, induces a vector bundle structure on $\pi_2:\R\times T^2M\to \R\times M$ and integral curves of $\xi\in\G(\pi_2)$ (the zero section) are autoparallels of $N$. \\

\textbf{ii.} Let $\pi_2:\R\times T^2M\to \R\times M$ admits a vector bundle structure isomorphic to $\R\times TM\oplus \R\times TM$. Then on $M$ a time dependent nonlinear connection can be defined.
\end{Pro}

Let $\partial_i$, $i=1,2,3$, be derivative with respect to the $i$'th variable. Then we have the following proposition which we leave the proof as an exercise.
\begin{Pro}\label{pro semispray yields nonlinear connection}
Let $S$ be a time dependent semispray on $\R\times M$. Then the families $\{  (N_\a^0=0,N_\a^1=\partial_3G_\a)\}_{\a\in I}$ and $\{  (N_\a^0=\partial_1G_\a,N_\a^1=\partial_3G_\a)\}_{\a\in I}$ are two nonlinear connections on $\R\times M$.
\end{Pro}

%
%
%
%
%
%
\section{time dependent Lagrangians}
In this section we review the concepts of time dependent Lagrangian and Lagrangian vector fields.
Then we derive the induced geometric structures on $\R\times T^2M$ where $M$ is a (possibly infinite dimensional) manifold endowed with a regular time dependent Lagrangian. To this end we follow the notation of \cite{Jer-Cher}.

The calculations of the fundamental 2-form of a time dependent Lagrangian and the Lagrangian vector fields are known in finite dimensional case (see e.g. \cite{Mir-Anas}) but we could not find any reference which contains the infinite dimensional version.

We start   with a definition from \cite{Jer-Cher}.

\begin{Def}
A bilinear continuous map $B:\E\times \E\to \R$ is called weakly nondegenerate if for any $y\in\E$ the map $B^b:\E\to\E^*$, $B^b(y)$ defined by $B^b(y)z=B(y,z)$ is injective. We call $B$   nondegenerate (or strongly   nondegenerate) if $B^b$ is an isomorphism \cite{Jer-Cher}.
\end{Def}
Note that if $\E$ is a finite dimensional Banach space, then there is no difference between strong and weak nondegeneracy.

\begin{Def}
A time dependent Lagrangian on $M$ is a differentiable map $L:\R\times TM\to \R$. The Lagrangian $L$ is called regular if $\partial_3^2L(t,x,y)$, $(t,x,y)\in\R\times TM$, is a nondegenerate bilinear map.
\end{Def}
The Lioville vector field $\G$ on $\R\times TM$ is defined by
\begin{eqnarray*}
\G : \R\times TM   &\to&   T(\R\times TM)\\
(t, x,y)  &\mto&   (t,x,y;0,0,y)
\end{eqnarray*}
and the  canonical tangent structure on $J$ on  $\R\times TM$ is
\begin{eqnarray*}
J:   T(\R\times TM)   &\to&   T(\R\times TM)\\
(t,x,y;s,z,w)   &\mto&   (t,x,y;0,0,z)
\end{eqnarray*}
For $L$ consider the canonical one form $\theta_L:=dL\o J$ on $\R\times TM$. Then
locally we have the formula
\begin{equation*}
\theta_L(t,x,y)(s,z,w)=dL(t,x,y)(0,0,z)=\partial_3L(t,x,y)z.
\end{equation*}
In the finite dimensional case with the local coordinate $(t,x^i,y^i)$ of $\R\times TM$ we have, $\G=y^i\frac{\partial}{\partial y^i}$ and $J=\frac{\partial}{\partial y_i}\otimes dx^i$ and $\theta_L$ is given by the expression $\theta_L=\frac{\partial L}{\partial y^i}dx^i$. Moreover $L$ is regular if the matrix $(\frac{\partial^2L}{\partial y^i\partial y^j})$ has the constant rank $n=dim M$.

The canonical symplectic form on $\R\times TM$ is  $\omega_L:=-d\theta_L$ which for any $v=(t,x,y)\in \R\times TM$ and $(s_i,z_i,w_i)\in T_v(\R\times TM)$, $i\in\{1,2\}$, $\omega_L$ is given by
\begin{eqnarray}\label{eq omega L}
\nonumber\omega_L(t,x,y) \Big(    (s_1,z_1,w_1) , (s_2,z_2,w_2)   \Big)  &=&  \partial_1\partial_3L(v)[z_1,s_2]-\partial_1\partial_3L(v)[z_2,s_1]\\
\nonumber&& + \partial_2\partial_3L(v)[z_1,z_2] -\partial_2\partial_3L(v)[z_2,z_1] \\
 &&+\partial_3^2L(v)[z_1,w_2] -\partial_3^2L(v)[z_2,w_1]
\end{eqnarray}
\begin{Rem}
In the autonomous case we have $\theta_L(x,y)(z,w)=\partial_2L(x,y)z$ and
\begin{eqnarray*}
\omega_L(x,y) \Big(    (z_1,w_1) , (z_2,w_2)   \Big)  &=& \partial_1\partial_2L(x,y)[z_1,z_2] -\partial_1\partial_2L(x,y)[z_2,z_1] \\
 &&+\partial_3^2L(x,y)[z_1,w_2] -\partial_2^2L(x,y)[z_2,w_1]
\end{eqnarray*}
which can derived from (\ref{eq omega L}) by setting $s=0$ (see also \cite{Jer-Cher} section 5.1).
\end{Rem}
%
%
%
%
\begin{Pro}\label{Pro Lagrangian vector field}
Let $L$ be a regular time dependent Lagrangian and  $E_L:=\G L-L$ be the energy of $L$ and $\Omega_L=\omega_L+dE_L\wedge dt$. Then the following statements hold true.\\
\textbf{i.} There is a unique semispray  $Z\in \mathfrak{X}(\R\times TM)$ for which
\begin{equation*}
i_Zdt=1 ~~\textrm{and}~~ i_Z\Omega_L=0.
\end{equation*}
\textbf{ii.} A curve $c(s)$ in $M$ is a geodesic of $Z$ if and only if it satisfies the Euler Lagrange's equation
\begin{equation}\label{eq lagrange}
\frac{d}{ds}\{  \partial_3L\big(   s,c(s), c'(s)   \big) [c'(s)]  \}=\partial_2L\big(   s,c(s), c'(s)   \big) [c'(s)].
\end{equation}
\end{Pro}
\begin{proof}
The energy function  $E_L$ locally maps $(t,x,y)$ to $\partial_3L(t,x,y)y-L(t,x,y).$
Let $Z$ be a vector filed on $\R\times TM$. Then there are smooth maps $X_0:{\R\times TU} \to \R$ and $X_i:{\R\times TU} \to \E$, $i=1,2$, such that for any $v=(t,x,y)\in \R\times TU$
\begin{equation*}
Z(t,x,y)=(t,x,y;X_0(v),X_1(v),X_2(v)).
\end{equation*}
The equation $i_Zdt=1$ implies that $X_0(t,x,y)=1$. Moreover
\begin{eqnarray*}
&&i_Z\Omega_L(v)\big( s,z,w   \big)=\Omega_L(v)\Big(Z(v) , (s,z,w)   \Big)\\
& =  &  \omega_L(v)\Big(Z(v) , (s,z,w)   \Big)+sdE_L(Z(v))-dE_L(s,z,w)\\
&=&  \partial_1\partial_3L(v)[X_1(v),s]-\partial_1\partial_3L(v)[z,1] + \partial_2\partial_3L(v)[X_1(v),z] \\
&&-\partial_2\partial_3L(v)[z,X_1(v)] +\partial_3^2L(v)[X_1(v),w] -\partial_3^2L(v)[z,X_2(v)]\\
&&+ sdE_L(Z(v)) -  \partial_1\partial_3L(v)[y,s] - \partial_2\partial_3L(v)[y,z] -\partial_3^2L(v)[y,w]\\
&& +\partial_1L(v)s+\partial_2L(v)z.
\end{eqnarray*}
Setting $s=0,z=0$ we conclude that
\begin{eqnarray*}
\partial_3^2L(v)[X_1(v),w]-\partial_3^2L(v)[y,w]=0
\end{eqnarray*}
and regularity of $L$ implies that $X_0(t,x,y)=y$. This last means that $Z$ is a second order differential equation. Applying $X_1(v)=y$ we obtain
\begin{eqnarray*}
i_Z\Omega_L(v)\big( s,z,w   \big)&=&  -\partial_1\partial_3L(v)[z,1] -\partial_2\partial_3L(v)[z,y]  -\partial_3^2L(v)[z,X_2(v)] \\
&&+ sdE_L(Z(v))   +\partial_1L(v)s+\partial_2L(v)z.
\end{eqnarray*}
Setting $s=0$ the above equation implies that
\begin{eqnarray}\label{eq X2}
\partial_1\partial_3L(v)[z,1] +\partial_2\partial_3L(v)[z,y]  +\partial_3^2L(v)[z,X_2(v)]  =\partial_2L(v)z.
\end{eqnarray}
Since $L$ is regular then the above equation implies  that
\begin{eqnarray*}
X_2(v)=(\partial_3^2L(v))^{-1}  [   \alpha(v)]
\end{eqnarray*}
where $\alpha(v)\in \E^*$ and
\begin{eqnarray*}
\alpha(v)z=\partial_2L(v)z-\partial_1\partial_3L(v)[z,1]-\partial_2\partial_3L(v)[z,y].
\end{eqnarray*}
%
%
%
%
\textbf{ii.} If $z=y=c'(s)$ and $c''(s)=X_2(s,c(s),c'(s))$ then equation (\ref{eq X2}) implies that
\begin{equation}
\frac{d}{ds}\{  \partial_3L\big(   s,c(s), c'(s)   \big) [c'(s)]  \}=\partial_2L\big(   s,c(s), c'(s)   \big) [c'(s)]
\end{equation}
as desired.
\end{proof}
%
%
%
%
Using the Lagrangian vector field we state the following theorem which proposes a rich vector bundle structure for the second order time dependent tangent bundle.
%
%
%
%
\begin{The}\label{The sections of T2M are motions of Lagrangian}
Let $L$ be a regular time dependent Lagrangian. Then $L$ induces a vector bundle structure on $\pi^2_L:\R\times T^2M\to \R\times M$ such that the integral curves of the zero section of $\pi^2_L$ are solutions of the Euler Lagrange equation (\ref{eq lagrange}).
\end{The}
\begin{proof}
Consider the vector bundle structure proposed by theorem   \ref{theo vb st by spray} with $G_\a=X_2$ proposed in proposition \ref{Pro Lagrangian vector field}. Then theorem \ref{theo vb st by spray} guaranties that $\pi^2_L:\R\times T^2M\to \R\times M$ admits a vector bundle structure with fibres isomorphic to $\E\times \E$ and the structure group $GL(\E\times\E)$. Moreover the integral curves of the zero section $\xi:\R\times M\to \R\times T^2M$  are geodesics of $Z$. Now the last part of proposition \ref{Pro Lagrangian vector field} implies that the integral curves of $\xi$ are solutions of the Euler Lagrange equation (\ref{eq lagrange}).
\end{proof}
%
%
%
%

Motivated by \cite{Mir-Anas} chapter 13, we introduce another canonical semispray on $\R\times TM$.
%
%
%
%
\begin{Pro}
If $L$ is a regular time dependent Lagrangian, then the followings hold true.\\
\textbf{i.} The family
\begin{eqnarray}\label{eq G from lagrangian}
G_\a(v)  &=&   \pa_3^2L_\a(v)^{-1}\{  \pa_2\pa_3L_\a(v)[.,y] -\pa_2L_\a(v)  \} ~~~;~\a\in I
\end{eqnarray}
 $v=(t,x,y)\in\R\times U_\a\times \E$,  defines a semispray on $\R\times TM$ whose coefficients depend only on $L$.\\
\textbf{ii.} There exists a nonlinear connection $N_L$ on $\R\times TM$ which depends only on the Lagrangian $L$.
The coefficients of $N_L$ are
\begin{eqnarray*}
N^0_\a(v)  &=&   \partial_3^2L_\a(v)^{-1}\{ \partial_1\partial_3L_\a(v)\}, \\
N_\a^1(v)   &=&   \pa_3G_\a(v).
\end{eqnarray*}
\end{Pro}
\begin{proof}
\textbf{i.} First we show that $\{G_\a\}_{\a\in I}$ satisfy the compatibility condition (\ref{eq compatibility condition for semisprays}).

Suppose that $v=(t,x,y)\in \R\times U_{\b\a}\times \E$, $s\in \R$ and $z,w\in\E$ be arbitrary elements and set  $v'=(t,\p_{\b\a}(x),d\p_{\b\a}(x)y)$. Then  we have the following compatibility conditions for partial derivatives of $L$.
\begin{eqnarray*}
\pa_1L_\a(v)s    &=&   \pa_1L_\b(v')s,\\
\pa_2L_\a(v)z   &=&    \pa_2L_\b(v')d\p_{\b\a}(x)z  +   \pa_3L_\b(v')d^2\p_{\b\a}(x)(y,z),\\
\pa_3L_\a(v)w   &=&   \pa_3L_\b(v')d\p_{\b\a}(x)w ,\\
\pa_1\pa_3L_\a(v)[w,s]   &=&   \pa_1\pa_3L_\b(v')[d\p_{\b\a}(x)w,s],\\
\pa_2\pa_3L_\a(v)[w,z]  &=&  \pa_2\pa_3L_\b(v')[d\p_{\b\a}(x)w,d\p_{\b\a}(x)z]\\
&&+   \pa_3^2L_\b(v')[d\p_{\b\a}(x)w,d^2\p_{\b\a}(x)(z,y)] \\
&&+  \pa_3L_\b(v')[d^2\p_{\b\a}(x)(w,z)]
\end{eqnarray*}
and
\begin{eqnarray}\label{eq partial  3 3}
\pa_3^2L_\a(v)[w,z]  =   \pa_3^2L_\b(v')[d\p_{\b\a}(x)w,d\p_{\b\a}(x)z]
\end{eqnarray}
where $L_\a=L\o\s_\a^{-1}$ and $L_\b=L\o\s_\b^{-1}$ and $U_{\a\b}\neq\emptyset$.

However by (\ref{eq G from lagrangian}) and the above equations  we have
\begin{eqnarray*}
&&  \pa_3^2 L_\a(v)[G_\a(v),z]= \pa_2\pa_3L_\a(v)[z,y] -\pa_2L_\a(v)z\\
&=&      \pa_2\pa_3L_\b(v')[  d\p_{\b\a}(x)z,d\p_{\b\a}(x)y]  + \pa_3^2L_\b(v')[  d\p_{\b\a}(x)z,d^2\p_{\b\a}(x)(y,y)]\\
&&- \pa_2L_\b(v')d\p_{\b\a}(x)z\\
&=&\pa_3^2L_\b (v')           [  G_\b(v')+d^2\p_{\b\a}(x)(y,y)  ,  d\p_{\b\a}(x)z  ].\\
\end{eqnarray*}
On the other hand equation (\ref{eq partial  3 3}) implies that
\begin{equation*}
\pa_3^2 L_\a(v)[G_\a(v),z]=   \pa_3^2L_\b (v')           \Big[  d\p_{\b\a}(x)[G_\a(v)]  ,  d\p_{\b\a}(x)z  \Big].
\end{equation*}
As a result of the last two equations we get
\begin{equation*}
 d\p_{\b\a}(x)[G_\a(v)]=G_\b(v')  +  d^2\p_{\b\a}(x)(y,y)
\end{equation*}
that is the family $\{G_\a\}_{\a\in I}$ defines a semispray on $\R\times TM$.

\textbf{ii.} Is a result of the first part  and proposition \ref{pro semispray yields nonlinear connection}.
\end{proof}
%
%
%
%

%
%
%
%
%
%
\subsection{Motion in a time dependent potential field}
In this section for a Riemannian manifold $(M,g)$ and its canonical Lagrangian $L(x,y)=\frac{1}{2}g(x)$ $(y,y)$ with a time dependent potential, we propose a vector bundle structure for $\R\times T^2M$ which encodes the geometric structures of this mechanical system through its trivializations.

Let $(M,g)$ be a Riemannian manifold. Define the Lagrangian $L_g(x,y)=\frac{1}{2}g(x)(y,y)$. In this case the equation (\ref{eq X2}) reduces to
\begin{eqnarray*}
g(x)(K_2(x,y),z)=\frac{1}{2}\pa_xg(x)(y,y)z-\pa_xg(x)(y,z)y
\end{eqnarray*}
where $Z_g(t,x,y)=(t,x,y,1,y,K_2(x,y))$ is the Lagrangian vector field of $L_g$ (see e.g. \cite{Jer-Cher}, p.107 or \cite{Lang}, p.194).
Consider the differential map  $U:\R\times TM\to \R$ which is known as (the time dependent) potential.
\begin{Pro}
For the Lagrangian
\begin{eqnarray}\label{eq potential field lagrangian}
L(t,x,y)=\frac{1}{2}g(x)(y,y)-U(t,x)
\end{eqnarray}
the associated second order vector field $Z_U$  is given by $Z_U(t,x,y)=(t,x,y,1$, $y,X_2(t,x,y))$ where
\begin{equation*}
X_2(t,x,y)=K_2(x,y)-grad U(t,x)
\end{equation*}
and $K_2$ is the canonical spray determined by the autonomous Lagrangian $L(x,y)=\frac{1}{2}g(x)(y,y)$ and $grad U(t,x)$ is  gradient of $U$.
\end{Pro}
\begin{proof}
It suffices to write equation (\ref{eq X2}) for the Lagrangian (\ref{eq potential field lagrangian}). More precisely, after some calculations we get
\begin{eqnarray*}
g(x)(X_2(t,x,y),z)  &=&   \frac{1}{2}\partial_xg(x)(y,y)z-\partial_xg(x)(y,z)y-\partial_2U(t,x)z\\
&=&   g(x)(K_2(x,y),z)-g(x)(gradU(t,x),z)
\end{eqnarray*}
and consequently
\begin{equation}\label{spray in a potential field}
X_2(t,x,y)=K_2(x,y)-gradU(t,x).
\end{equation}
\end{proof}
\begin{Rem}
Following the formalism of theorem \ref{The sections of T2M are motions of Lagrangian}, one can endow $\pi^2_U:\R\times T^2M\to \R\times M$ with fibres isomorphic to $\E\times \E$ such that the integral curves of the zero section of $\pi_U^2$ are motions of the Lagrangian systems with the potential $U$ (that is equation (\ref{eq lagrange}) with $L$ given by (\ref{eq potential field lagrangian})).
\end{Rem}
%
%

\subsection{External forces } In this section we shall try to reveal  the vector bundle structures induced by time dependent external forces and holonomic constraints on  second order time dependent  tangent bundles.

We remind from \cite{Bullo-Lewis} that an external force  on $M$ is a differentiable map $F:\R\times TM\to T^*M$ with the property that $F(t,x,y)\in T_x^*M$ for each $v=(t,x,y)\in \R\times TM$.

Every external force induces a horizontal one form $\omega_F$ on $\R\times TM$ given by
\begin{equation*}
\omega_F(t,x,y)(t,x,y;s,z,w)=F(t,x,y)(x,z)
\end{equation*}
for any $(t,x,y;s,z,w)\in T_v (\R\times TM)$.
Remind that a horizontal 1-form on $\R\times TM$ is a 1-form $\omega$ for which $\omega(Y)=0$ for all vertical vector field $Y:\R\times TM\to V\pi$.
Denote by $\Omega^1_{hor}(\R\times TM)$ the space of horizontal 1-forms on $\R\times TM$ and let $\mathfrak{X}^v(\R\times TM)$ be the space of all vertical vector fields on $\R\times TM$.

\begin{Pro}\label{pro horizontal forms and vertical vector fields}
Let $L$ be a regular time-dependent Lagrangian. Then there is a one to one correspondence between horizontal 1-forms and vertical vector fields on $R\times TM$.
\end{Pro}
\begin{proof}
Let $\omega\in\Omega^1_{hor}(\R\times TM)$. Consider the unique vertical  vector field
\begin{eqnarray*}
Y:\R\times TM  \to  V\pi~~~;~~~~((t,x,y)   \mto   (t,x,y;0,0,Y_2(t,x,y))
\end{eqnarray*}
given by the relation $\omega=-i_Y\omega_L$. More precisely
\begin{eqnarray*}
\omega(t,x,y)(s,z,w)&=&-\omega_L \big(   (0,0,Y_2(t,x,y)) ,   (s,z,w)  \big)\\
&=&  \partial_3^2L(t,x,y)[z,Y_2(t,x,y)]
\end{eqnarray*}
that is
\begin{equation*}
Y_2(t,x,y)=\partial_3^2L(t,x,y)^{-1}\big( \omega(t,x,y)  \big).
\end{equation*}
Conversely suppose that $Y$ be a vertical vector field on $\R\times TM$. Define $\omega$ by
$\omega=-i_Y\omega_L$. It is a easily seen that $\omega$ belongs to $\in\Omega^1_{hor}(\R\times TM)$.
\end{proof}
\begin{Cor}
As a consequence of the above proposition, for any external force $F:\R\times TM\to T^*M$ we have a vertical vector field $Y_F$ given by
\begin{eqnarray*}
Y_F(t,x,y)=\Big( t,x,y;0,0,\partial_3^2L(t,x,y)^{-1}[F(t,x,y)\o T\tau_M]\Big)
\end{eqnarray*}
where $\tau_M:\R\times TM \to M$ projects $(t,x,y)$ onto $x$.
\end{Cor}

\begin{Pro}
Let $X\in\mathfrak{X}(\R\times TM)$ and $i_Xdt=1$. Then $i_X(\omega_L-dE\wedge dt)$ is horizontal if and only if $X $ is a semispray.
\end{Pro}
\begin{proof}
Let $X\in\mathfrak{X}(\R\times TM)$ be given by $X(t,x,y)=\big(  X_0(v),X_1(v),X_2(v) \big)$, $v=(t,x,y)\in\R\times TM$, $i_Xdt=1$ and $i_X(\omega_L-dE\wedge dt)$ be horizontal. Then the condition $i_Xdt=1$ implies that $X_0=1$. Moreover for every $(v;0,0,w)\in T_v(\R\times TM)$  equation (\ref{eq omega L}) implies that
\begin{eqnarray*}
\omega_L\big(  (1,X_1(v),X_2(v) )  ,   (0,0,w)  \big)&-&dE\wedge dt\big(   (1,X_1(v),X_2(v))  , (0,,0,w)  \big)\\
&=& \partial_3^2L(v)[X_1(v),w]-\partial_3^2L(v)[y,w]=0
\end{eqnarray*}
that is $X_1(t,x,y)=y$ and therefore $X$ is a semispay.

Conversely suppose that $X:\R\times TM \to T(\R\times TM)$ ; $v\mto (v$; $1,y,X_2(v))$ be a semispray. We show that $i_X(\omega_L-dE\wedge dt)$ is horizontal that is
\begin{equation*}
i_X(\omega_L-dE\wedge dt)(v;0,0,w)=0.
\end{equation*}
But this is a direct consequence of equation (\ref{eq omega L}). In fact
\begin{eqnarray*}
i_X(\omega_L-dE\wedge dt)(v;0,0,w)  &=&   \omega_L\big(  (1,X_1(v),X_2(v) )  ,   (0,0,w)  \big)\\
&&-dE\wedge dt\big(   (1,X_1(v),X_2(v))  , (0,,0,w)  \big)\\
&=&  \partial_3^2L(v)[y,w]-\partial_3^2L(v)[y,w]=0
\end{eqnarray*}
as we required.
\end{proof}
%
%
%
Now, suppose that $L$ be a time dependent Lagrangian on $M$ and $F$ be the external force of our system. In the next proposition we show that this system is governed by a vertical vector field which depends on $L$ and $F$.
\begin{Pro}\label{pro motions in lagrangian system with force}
Let $X\in\mathfrak{X}(\R\times TM)$ and $i_Xdt=1$ and $F:\R\times TM\to T^*M$ be an external force. Then \\
\textbf{i}. $i_X(\omega_L-dE\wedge dt)-i_Y\omega_L=0$ if and only if $X=Z+Y$ where $Z$ is the canonical semispray induced by $L$ and $Y$ is the vertical vector field associated with $F$.\\
\textbf{ii.} The geodesics  of $X$ are motions of the system
\begin{equation}\label{eq of motion with external force}
\frac{d}{ds}\{  \partial_3L\big(   s,c(s), c'(s)   \big) [c'(s)]  \}=\partial_2L\big(   s,c(s), c'(s)   \big) [c'(s)]-Y_2\big(   s,c(s), c'(s)   \big)
\end{equation}
\end{Pro}
\begin{proof}
\textbf{i}. Let $Y$ be a vertical vector field on $\R\times TM$ associated with the external force $F$ and suppose that $X\in\mathfrak{X}(\R\times TM)$ and
$i_X(\omega_L-dE\wedge dt)-i_Y\omega_L=0$ and $i_Xdt=1$. Then, for any $v=(t,x,y)\in\R\times TM$ and $(v;s,z,w)\in T_v(\R\times TM)$ we have
\begin{eqnarray*}
&&\omega_L(v)\big(   (1,y,X_2(v))  ,   (s,z,w)  \big)   -     (dE\wedge dt)(\big(   (1,y,X_2(v))  ,   (s,z,w)  \big)  \\
&&-\omega_L \big(   (0,0,Y_2(v))  ,   (s,z,w)  \big)\\
&=&   \partial_1\partial_3L(v)[y,s]-\partial_1\partial_3L(v)[z,1] - \partial_2\partial_3L(v)[z,y]  -\partial_3^2L(v)[z,X_2(v)] \\
&&  +  sdE(X(v))  - \partial_1\partial_3L(v)[y,s] +  \partial_1L(v)s + \partial_2L(v)z +\partial_3^2L(v)[y,Y_2(v)]\\
&=&0.
\end{eqnarray*}
Setting $s=0$ we get
\begin{eqnarray*}
\partial_3^2L(v)[z,X_2(v)]  &=&  -\partial_1\partial_3L(v)[z,1] - \partial_2\partial_3L(v)[z,y]  + \partial_2L(v)z  \\
&&  +  \partial_3^2L(v)[y,Y_2(v)].
\end{eqnarray*}
As a  consequence of the above equation and equation (\ref{eq X2}) we  have
\begin{eqnarray*}
\partial_3^2L(v)[z,X_2(v)]  =   \partial_3^2L(v)[y,Y_2(v)+Z_2(v)].
\end{eqnarray*}
Since $L$ is a regular Lagrangian then $X_2=Y_2+Z_2$ as we  required.

Conversely suppose that $X=Y+Z$ then, clearly $i_Xdt=1$ and $i_X(\omega_L-dE\wedge dt)-i_Y\omega_L=0$.

\textbf{ii}. Proof of part two is a direct consequence of proposition \ref{Pro Lagrangian vector field} and part \textbf{i}.
\end{proof}
%
%
%
With the notations as above we have the following corollary.
\begin{Cor}
Let $(M,L)$ be a time dependent Lagrangian system and $F$ be an external force. Then the semispray $X=Y+Z$ induces a vector bundle structure on $\R\times T^2M$ over $\R\times M$, for which the integral curves of the zero section are motions of the system.
\end{Cor}
%
%
\subsection{Holonomic constraints}
Let $(M,g)$ be a Riemannian manifold (possibly infinite dimensional) and $F$ be a time dependent  external force on $M$. Consider the Lagrangian $L(v)=\frac{1}{2}g(v,v)$. In this section we study the geometric structures induced by holonomic constraints on $\R\times T^2M$.
\begin{Def}
A holonomic constraint on  $(M,g,F)$ is a submanifold $N$ of $M$.
\end{Def}

A reaction force of the holonomic constraint $N\subseteq M$ is given by $R:\R\times TN\to T^*M$. The reaction force is said to be perfect if $\mu^{-1}(R(t,v))\in  (T_xN)^{\perp}$ for any $(t,v)\in\R\times T_xN$ where
\begin{eqnarray*}
\mu:TM\to T^*M; ~(x,y)\mto \mu(x,y)
\end{eqnarray*}
and $\mu(x,y)(x,z)=g(x)(y,z)$ (see e.g. \cite{God-Nat} pp. 174-179).
Suppose that $R:\R\times TN\to T^*M$ be the perfect reaction force given by
\begin{eqnarray}\label{eq reaction force}
R(t,v)=\Big(t,\mu(B(v,v)) \Big)  -  \Big(t,\mu\big( \mu^{-1}  \big(  F(t,v)   \big)^\perp \big) \Big)
\end{eqnarray}
where $B$ is the second fundamental form of $N$ (for the autonomous case see \cite{God-Nat} chapter 5 section 2).

Let $(N,g|_N,F|_N)$  denote the mechanical system $N$ with the Lagrangian  induced by $g|_N$ and the external force $F|_N$. Moreover suppose that $(M,g,F+R)$ denote the manifold $M$ with the Lagrangian $L(x,y)=\frac{1}{2}g(x)(y,y)$, $(x,y)\in TM$ and the external force $F+R$.

\begin{The}
With the  assumptions as above $(N,g|_N, F|_N)$ and  $(M$, $g,F+R)$ induce vector bundle structures on $\R\times T^2N$ and $\R\times T^2M$ respectively. Moreover the integral curves of the zero section of $(\pi^2_{F|_N},\R\times T^2N,\R\times N)$ are integral curves of $\xi_M$ (the zero section) of the bundle $(\pi^2_{F+R},\R\times T^2M,\R\times M)$
\end{The}
\begin{proof}
Consider the map $R_M(t,w)=R(t,p_N(w))$ where $p_N:TM\to TN$ is orthogonal projection which maps $w=w^\top+w^\perp$ onto $w^\top$.  Then, using proposition \ref{pro horizontal forms and vertical vector fields}, $R_M$ induces a vertical vector field, say $W$, on $M$. It is easy to check that the semispray of the system $(M,g,F+R)$ is given by $X=K+Y+W$ where $K$ is the spray of the metric $g$ and $Y$ is the induced vertical vector field by $F$ and  proposition \ref{pro motions in lagrangian system with force}. Moreover the time dependent semispray of the system $(N,g|_N,F|_N)$ is $X|_N=K|_N+Y|_N$. Since the motions of $X|_N=K|_N+Y|_N$ are motions of  $(M,g,F+R)$ (see e.g theorem 2.7 page 176, \cite{God-Nat}), then   the integral curves of the zero section of $\pi^2_{F|_N}:\R\times T^2N\to \R\times N$ are integral curves of $\xi_M$ (the zero section) of the vector bundle  $\pi^2_{F+R}:\R\times T^2M\to \R\times M$.

\end{proof}

%
%
%
%
%
\subsection{Invariant Lagrangians and second order vector bundle morphisms}
In this section first we  reveal the relations between Invariant Lagrangians and related semisprays. Then we will show that the vector bundle structure on $\R\times T^2M$ remains invariant if we substitute $L$ with $L\o Tf$ where $f:M\to M$ is a diffeomorphism with $L\o (id_\R\times Tf)=L$.

Let $M$ and $N$ be two  manifolds, $f:M\to N$ be a diffeomorphism and $L_N:\R\times TN\to \R$ be a regular time dependent Lagrangian. Consider the regular Lagrangian $L_M=L_N\o (id_\R\times Tf)$ on $M$.
\begin{Def}
The semisprays $Z_M\in\mathfrak{X}(TM)$ and $Z_N\in\mathfrak{X}(TN)$ are called $f$-related if $TTf\o Z_M=Z_N\o Tf$.
\end{Def}
If we identify $U\subseteq M$ and $V\subseteq N$ with their images in the model spaces $\E$ and $\F$ respectively, then
$TTf\o Z_M|_U=Z_N\o Tf|_U$ if and only if
\begin{eqnarray}\label{eq related sprays}
X_2^N(t,f(x),df(x)y)=d^2f(x)(y,y)+df(x)X_2^M(t,x,y)
\end{eqnarray}
where $(t,x,y)\in\R\times U\times \E,$ $Z_M(t,x,y)=(t,x,y;1,y,X_2^M(t,x,y))$ and $Z_N(t,x,y)=(t,x,y;1,y, X_2^N(t,x,y))$.
\begin{Pro}\label{pro invariant lagrangians and semisprays}
If $L_M$ and $L_N$ are as above with the semisprays $Z_M$ and $Z_N$ respectively then, $Z_M$ and $Z_N$ are $f$-related.
\end{Pro}
\begin{proof}
We show that, locally, the induced  semisprays by $L$ and $L\o (id_\R\times Tf)$ satisfy the compatibility condition (\ref{eq related sprays}).

Since $L_M=L_N\o (id_\R\times Tf)$ then, $L_M(v)=L_N(v')$ where $v=(t,x,y)$ and $v'=(t,f(x),df(x)y)$. It is not hard to check that
\begin{eqnarray*}
\pa_2L_M(v)z   &=&   \pa_2L_N(v')df(x)z + \pa_3L_N(v')d^2f(x)(z,y)\\
\pa_3L_M(v)z   &= &  \pa_3L_N(v')df(x)z,\\
\pa_1\pa_3L_M(v)[z,1]   &=&    \pa_1\pa_3L_N(v')[df(x)z,1],\\
\pa_2\pa_3L_M(v)[z,w]   &=&   \pa_2\pa_3L_N(v')[df(x)z,df(x)w]  + \pa_3L_N(v')d^2f(x)(z,w)\\
&&+\pa_3\pa_3L_N(v')[df(x)z,d^2f(x)(y,w)]
\end{eqnarray*}
and
\begin{eqnarray*}
\pa_3\pa_3L_M(v)[z,w]   &=&   \pa_3\pa_3L_N(v')[df(x)z,df(x)w].
\end{eqnarray*}
Using the above equations and  (\ref{eq X2}) for $L_M$ and $L_N$ we get
\begin{eqnarray*}
&&  \pa_3\pa_3L_M(v)[X_2^M(v),z]=   -\pa_1\pa_3L_M(v)[z,1] -\pa_3\pa_3L_M(v)[z,y] + \pa_2L_M(v)z\\
&=&  -\pa_1\pa_3L_N(v')[df(x)z,1]  -  \pa_3L_N(v')d^2f(x)(z,y)\\
&& -\pa_3\pa_3L_N(v')[df(x)z,d^2f(x)(y,y)] +\pa_2L_N(v')df(x)z + \pa_3L_N(v')d^2f(x)(z,y)\\
&=&  \pa_3\pa_3L_N(v')[X_2^N(v'),df(x)z]-\pa_3\pa_3L_N(v')[df(x)z,d^2f(x)(y,y)]\\
&=& \pa_3\pa_3L_N(v')[X_2^N(v')-d^2f(x)(y,y),df(x)z].
\end{eqnarray*}
Since $L_N$ is regular we obtain
\begin{eqnarray*}
df(x)X_2^M(t,x,y)= X_2^N(t,f(x)df(x)y) - d^2f(x)(y,y)
\end{eqnarray*}
that is $Z_M$ and $Z_N$ are $f$-related.
\end{proof}
For any map $f:M\to N$, the induced map $T^2f:\R\times T^2M\to \R\times T^2N$; $(t,[\g,x]_2)\mto (t,[f\o\g,f(x)]_2)$ generally is not vector bundle morphism (see e.g. \cite{Dod-Gal-Vass, ali, iso Osck}). Following  \cite{Dod-Gal-Vass, ali, iso Osck} one can show that if the semisprays $Z_M\in\mathfrak{X}(\R\times TM)$ and $Z_N\in\mathfrak{X}(\R\times TN)$ are $f$-related, then %
\begin{eqnarray*}
T^2f:(\R\times T^2M,Z_M,\R\times M)  &\to&  (\R\times T^2N,Z_N,\R\times N)\\
(t,[\g,x]_2)  &\mto&    (t,[f\o\g,f(x)]_2)
\end{eqnarray*}
becomes a vector bundle morphism. As a consequence of the above discussion, with the assumptions as in proposition \ref{pro invariant lagrangians and semisprays} we have the following theorem.
%
%
%
\begin{The}
If $f:M\to N$ is a diffeomorphism then,  $T^2f$ is an vector bundle isomorphism.
\end{The}
\begin{proof}
Since $L_M$ and $L_N$ are $f$-related, then proposition \ref{pro invariant lagrangians and semisprays} implies that $Z_M$ and $Z_N$ are $f$-related. This last means that  $T^2f$ is a vector bundle morphism with the inverse $T^2f^{-1}$.
\end{proof}
\begin{Rem}
Let $G$ be a Lie group acting on a manifold $M$
via the map $A:G\times M\to M$. The Lagrangian $L:\R\times TM\to \R$ is called $G$-invariant if
\begin{equation*}
L\o T A_g=L ~;~g\in G
\end{equation*}
where $A_g:M\to M$ maps $m\in M$ onto $A(g,m)$. As a consequence of the above results, for any $g\in G$, the map $(T^2A_g, A_g):(\R\times T^2M, L, M)\to (\R\times T^2M, L, M) $ is a vector bundle isomorphism.

\end{Rem}
The above discussion says (roughly speaking), that
"\emph{Invariant Lagrangians induce invariant vector bundle structures (up to isomorphism) on $\R\times  T^2M$}".

%
%
%
%
\section{Applications and examples}
\subsection{$\R\times T^2\D^s_\mu$ and the motion of an incompressible fluid}
\begin{Examp}
Let $M$ be a compact Riemannian manifold filled with a perfect (incompressible, homogeneous and inviscid) fluid. As it was pointed out by Arnold \cite{Arnold}, the configuration space of this fluid is the group of volume preserving diffeomorphisms $\D_\mu^s$ and the motion of the fluid can be described by a curve $\eta:(-\eps,\eps)\to \D^s_\mu$ where $\eta_t(x)$ is the position of $x\in M$ at time $t$ and $\eta_0=id_M$. In this example, following \cite{Ebin-Marsden}, first we introduce  $\D_\mu^s$ and its geometric structure. Then we endow $\pi_{\D^s_\mu}^2:\R\times T^2\D^s_\mu \to\R\times \D^s_\mu$ with a suitable vector bundle structure such that the integral curves of its zero section yield a solution for the Euler equation
$$\left\{ \begin{array}{ll}  \frac{dv(t)}{dt}+\nabla_{v_t}v_t =  grad p_t  +f_t \\
div(v_t)=0  ,  v_t ~ \textrm{is given at} ~ t=0 ~\textrm{ and }~ div(f_t)=0
\end{array}\right.(E)$$
which describes the motions of the fluid.
In fact, (E) is the Euler equation for the incompressible  fluid with time dependent external force $f_t$
where $\nabla$ is the Levi-Civita connection and $p_t:M\to \R$ is the pressure.


In order to introduce the necessary geometric structures, we review some parts of \cite{Ebin-Marsden} and \cite{Marsden-Ebin-Fischer}.

Let $M$ and $N$ be compact manifolds. Then $H^s(M,N)=\{\eta:M\to N~;~ \eta ~\textrm{is}~ H^s\}$ is a Banach manifold. Note that $H^s$ means that $\eta $ and $\eta^{-1}$ and all their partial derivatives up to order "$s$" are square integrable.

Let $(M,g)$ be a compact Riemannian manifold of dimension $n$. For our purpose we assume that $M$ is without boundary. For $s>\frac{n}{2}+k$ suppose that $\mathcal{D}^s$ be the set of all bijective maps $\eta:M\to M$ such that $\eta$ and $\eta^{-1}$ are both $H^s$.  Then the Sobolev embedding theorem states that
$$H^s(M) \subseteq C^k(M)=\{\eta;~M\to M~;~~ \eta ~\textrm{ is a }~ C^k~ \textrm{diffeomorphism}\}$$
and the inclusion is continuous.

Suppose that $\pi:TM\to M$ be the the tangent bundle. For $\eta\in\D^s$ define
\begin{equation*}
T_\eta{\D^s}=\{f\in H^s(M,TM);~\pi\o f=\eta\}.
\end{equation*}
Then $T_\eta \D^s$ is a Hilbert space with $H^s$ topology. Let $(M,g)$ be the exponential map of $(M,g)$. For $\eta\in \D^s$ we define the local chart
\begin{eqnarray*}
\s_\eta :U_\eta\subseteq T_\eta\D^s &\to&  \D^s\\
f  &\mto& exp\o f
\end{eqnarray*}
where $U_\eta$ is neighborhood of zero section in $T_\eta\D^s$. More precisely $U_\eta$ is the set of all $f\in T_\eta \D^s$ such that $\|f(x)\|<$ the injectivity radius of $exp$ at $\eta(x)$. Then $\{U_\eta,\s_\eta)\}_{\eta\in\D^s}$ form an atlas for $\D^s$ (for more details see \cite{Ebin-Marsden}, \cite{Eli} or \cite{Marsden-Ebin-Fischer}.

Note that $T_{id}\D^s$ is the space of all $H^s$ vector fields on $M$ and
\begin{equation*}
T\D^2=\cup_{\eta\in\D^s}T_\eta\D^s.
\end{equation*}
Now, suppose that $\mu$ is the volume element on $(M,g)$. Then, it is known that
\begin{equation*}
\D^s_\mu=\{ \eta\in\D^s;~\eta^*\mu=\mu\}
\end{equation*}
is a closed submanifold of $\D^s$ and
\begin{eqnarray*}
T_{id}\D^s_\mu=\{X\in T_{id}\D^s; \delta X=0\}=\{X\in T_{id}\D^s; div X=0\}
\end{eqnarray*}
(see e.g. \cite{Ebin-Marsden}). Moreover the smooth projection $P_{id}:T_{id}\D^s\to T_{id}\D^s_\mu$ is simply projection on to the first summand of the known Hodge decomposition
\begin{equation*}
T_{id}\D^s=div^{-1}(0)\oplus grad \mathfrak{F}^{s+1}
\end{equation*}
where $\mathfrak{F}^{s+1}$ is the set of all $H^{s+1}$ functions on $M$. The above projection can be extended to a smooth map
\begin{eqnarray*}
P_\eta : T_\eta\D^s  &\to &   T_\eta \D^s_\mu\\
X  &\mto&   (P_{id}(X\o\eta^{-1}) )\o \eta
\end{eqnarray*}
for any $\eta\in \D^s_\mu$.

Consider the Lagrangian $L:T\D^s\to \R$; $X\mto \frac{1}{2}\langle X,X\rangle$ where
\begin{equation*}
\langle X,Y\rangle=\int_M g(\eta(x))\big(  X(x),Y(x) \big)d\mu ~~;~~~\forall X,Y\in T_\eta\D^s.
\end{equation*}
Then the canonical spray of $L$ is
\begin{eqnarray*}
\bar{Z}:T\D^s&\to & T^2\D^s\\
X  &\mto& Z\o X
\end{eqnarray*}
where $Z:TM\to TTM$ is the canonical metric spray on $(M,g)$ \cite{Ebin-Marsden}. Moreover
\begin{eqnarray*}
S:T\D^s_\mu&\to & T^2\D^s_\mu\\
X  &\mto& TP\o Z\o X
\end{eqnarray*}
is a smooth spray on $\D^s_\mu$ compatible with the right invariant metric $\langle,\rangle|_{\D^s_\mu}$.

As a consequence, $S(X)-\bar{Z}(X)$, $X\in T_\eta\D_\mu^s$, belongs to $grad\mathfrak{F}^{s+1}$ and it determines a vertical vector field on $T\D_\mu^s$. Hence there exists an $H^{s+1}$ function $p:M\to \R$ such that $S(X)-\bar{Z}(X)=(grad p )^l$
where $l$ stands for the vertical lift of $grad p$. The term '$grad p$' depends on $X$ and it measures the force (pressure) of the holonomic constraint  $\D^s_\mu\subseteq \D^s$ \cite{Ebin-Marsden}.

Now, suppose that $f:(a,b)\subseteq\R\to T_{e}\D_\mu^{s+k}$, $k\in\N$  and define
 $\mathrm{F}:\R\times T\D_\mu^s\to  T(\R\times T\D_\mu^s)$ by
\begin{equation*}
\mathrm{F}(t,\eta,v)=(t,\eta,v;0,0,f_t\o\eta).
\end{equation*}
Then $\mathrm{F}$ is a $C^k$ map (\cite{Ebin-Marsden}, theorem 11.2) and $\mathrm{F}$ is the external force of the system. Moreover  $S+\mathrm{F}$ is a semispray on $\D^s_\mu$. If we identify $U\subseteq\D^s_\mu$ with its image then, by theorem \ref{the vb strucures} we have the family of vector bundle trivializations
\begin{eqnarray*}
 \P^2:{\pi_{\D_\mu^s}^2}^{-1}(U)  &\to&   \R\times U\times \E_1\times \E_1\\
(t,[\g,\eta]_2)   &\mto&   \Big(    t, \eta ,\g'(0), \g''(0)   + P_e\Big(\Gamma(\eta)[\dot{\g}(t)\o\eta^{-1},\dot{\g}(t)\o\eta^{-1}]   \Big)\o\eta\\
&& +f_t\o\eta   \Big)
\end{eqnarray*}
for $\pi^2_{\D^s_\mu}:T^2\D_\mu^s\to \D^s_\mu$ where $\E_1=div^{-1}(0)$. Setting  $\E_2=grad\mathfrak{F}^{s+1}$ and $\E=\E_1\oplus \E_2=div^{-1}(0)\oplus grad\mathfrak{F}^{s+1}$ it is easily seen that the spray $\bar{Z}$ also induces a vector bundle structure on $(\pi^2_{\D^s},T^2\D^s,\D^s)$ with fibres isomorphic to $\E\times \E$. In this case the inclusion $i:\D^s_\mu\to \D^s$ induces a vector bundle morphism
$$T^2i:(T^2\D^s_\mu,\pi_{\D_\s^\mu}^2,\D_\mu^s)\hookrightarrow (T^2D^s|_{D^s_\mu},\pi^2_{\D^s},D^s_\mu).$$
Furthermore, the integral curves of the zero section $\xi:\D^s_\mu\to T^2\D^s_\mu$ yield the solutions of the Euler equation $(E)$ which describes the configuration of the fluid in $M$. More precisely, suppose that $\eta:(-\eps,\eps)\to \D^s_\mu$ be an integral curve of the zero section $\xi$. Define $X_t=\frac{d}{dt}\eta_t$ and  $v_t=X_t\o\eta_t^{-1}$. Then we have
\begin{eqnarray*}
\frac{dv(t)}{dt}  &\stackrel{*}=&   \frac{dX_t}{dt}\o\eta_t^{-1}-TX_t\o T\eta_t^{-1}\o X_t\o\eta_t^{-1}\\
&=&S(X_t)\o \eta_t^{-1}+ F(X_t)\o\eta_t^{-1} -TX_t\o T\eta_t^{-1}\o X_t\o\eta_t^{-1}\\
&=&-P_e (  \nabla_{v_t}v_t)_v^l  +  f_t\\
&=& -\nabla_{v_t}v_t +grad p_t  +  f_t
\end{eqnarray*}
where in $*$ and $**$ we used lemma 1.3.5 of \cite{Marsden-Ebin-Fischer}.

\end{Examp}
\begin{Examp}
Let $M$ be a smooth  manifold, $I\subseteq \R$  and $(g_t)_{t\in I}$ be a family of Riemannian metrics on $M$. Consider the Lagrangian
\begin{eqnarray*}
L:\R\times TM &\to \R\\
(t,x,y) &\mto &  \frac{1}{2}g(t,x)(y,y)
\end{eqnarray*}
Then the standard argument in Riemannian geometry proposes a family of Riemannian sprays $\{S_t\}_{t\in I}$ and theorem \ref{theo vb st by spray} leads us to a vector bundle structure on $\pi_2:\R\times T^2M\to M$.

\end{Examp}

%
%

\end{document}